\documentclass[12pt]{amsart}

\usepackage{amsmath}
\usepackage{amssymb}
\usepackage{amsfonts}
\usepackage{amsthm}
\usepackage{enumerate}
\usepackage{hyperref}
\usepackage{color}
\theoremstyle{plain}
\newtheorem{thm}{Theorem}[section]

\newtheorem{lem}[thm]{Lemma}
\newtheorem{cor}[thm]{Corollary}

\theoremstyle{definition}

\newtheorem{rem}[thm]{Remark}

\newtheorem{dfns-rems}[thm]{Definitions and Remarks}
\newtheorem{notas-rems}[thm]{Notations and Remarks}
\newtheorem{exmps-rems}[thm]{Examples and Remarks}

\makeatletter
\@namedef{subjclassname@2020}{%
  \textup{2020} Mathematics Subject Classification}
\makeatother
\begin{document}

% ------------------------------------------------------------------------

\title[Regularity of powers of monomial ideals]{Some inequalities regarding the regularity of powers of monomial ideals}

% ------------------------------------------------------------------------

\author[S. A. Seyed Fakhari]{S. A. Seyed Fakhari}

\address{S. A. Seyed Fakhari, Departamento de Matem\'aticas\\Universidad de los Andes\\Bogot\'a\\Colombia.}

\email{s.seyedfakhari@uniandes.edu.co}

% ------------------------------------------------------------------------

\begin{abstract}
We prove some inequalities regarding the Castelnuovo--Mumford regularity of symbolic powers and integral closure of powers of monomial ideals.
\end{abstract}

% ------------------------------------------------------------------------


% ------------------------------------------------------------------------


% ------------------------------------------------------------------------

\thanks{}

% ------------------------------------------------------------------------

\maketitle

%%%%%%%%%%%%%%%%%%%%%%%%%%%%%%%%%%%%%%%%%%%%%%%%%%%%%%%%%%%%%%%%%%%%%%%%%%

\section{Introduction} \label{sec1}

Let $\mathbb{K}$ be a field and $S = \mathbb{K}[x_1,\ldots,x_n]$  be the
polynomial ring in $n$ variables over $\mathbb{K}$. Suppose that $M$ is a graded $S$-module with minimal free resolution
$$0  \longrightarrow \cdots \longrightarrow  \bigoplus_{j}S(-j)^{\beta_{1,j}(M)} \longrightarrow \bigoplus_{j}S(-j)^{\beta_{0,j}(M)} \longrightarrow  M \longrightarrow 0.$$
The Castelnuovo--Mumford regularity (or simply, regularity) of $M$, denoted by ${\rm reg}(M)$, is defined as
$${\rm reg}(M)=\max\{j-i|\ \beta_{i,j}(M)\neq0\},$$
and it is an important invariant in commutative algebra and algebraic geometry. By convention, we set ${\rm reg}(M)=-\infty$, if $M=0$ is the zero module. Regularity of powers of monomial ideals has been studied by several researchers (see for example \cite{b''}, \cite{bbh}, \cite{bh1}, \cite{c'}, \cite{cht}, \cite{hntt}, \cite{htt1}, \cite{js}, \cite{s14}, \cite{sy1}).

Recently, Minh and Vu \cite[Lemma 1.4]{mv} proved that for any squarefree monomial ideal $I$ and for each integer $m\geq 1$, we have$${\rm reg}(I)+m-1\leq \min\big\{{\rm reg}(I^m), {\rm reg}(I^{(m)})\big\},$$where $I^{(m)}$ denotes the $m$-th symbolic power of $I$. The goal of this paper is to extend the above inequality. We first prove in Theorem \ref{rrad} that for any monomial ideal $I$, we have$${\rm reg}(I)\geq {\rm reg}(\sqrt{I})+(\gamma(I)-1)h,$$where $h$ denotes the height of $I$ and $$\gamma(I):=\min\{{\rm deg}_{x_j}(u): 1\leq j\leq n, u\in G(I) \ {\rm and} \ x_j \ {\rm divides} \ u\}.$$

In Section \ref{sec3}, we focus on the regularity of symbolic powers of squarefree monomial ideals. As the main result of that section, we prove in Theorem \ref{sym} that if $I$ is a nonzero proper squarefree monomial ideal of $S$, then for any pair of positive integers $m$ and $k$ and for each integer $j$ with $m-k\leq j\leq m$, we have$${\rm reg}(I^{(km+j)})\geq {\rm reg}(I^{(m)})+(k-1)m+j.$$This in particular implies that$${\rm reg}(I^{(km)})\geq {\rm reg}(I^{(m)})+(k-1)m \ \ \ {\rm and} \ \ \ {\rm reg}(I^{(3)})\geq {\rm reg}(I^{(2)})+1$$ (see Corollary \ref{corsym}).

In Section \ref{sec4}, we study the regularity of integral closure of powers of monomial ideals. For a monomial ideal $I$, let $\overline{I}$ denote its integral closure and let $\mu(I)$ denote the cardinality of the set of minimal monomial generators of $I$. We prove in Theorem \ref{rnormal1} that if $\ell$ is the analytic spread of $I$, then for $s=\mu(\overline{I^{\ell-1}})!$ and for every integer $m\geq 1$, we have$${\rm reg}(\overline{I})\leq {\rm reg}(I^{sm})-\gamma(I)(sm-1).$$

As mentioned above, Minh and Vu \cite{mv} prove that for any monomial ideal $I$ and for each integer $m\geq 1$, the inequality ${\rm reg}(I^m)\geq {\rm reg}(I)+m-1$ holds. In Theorem \ref{rintc}, we generalize this result by showing that for any integrally closed monomial ideal $I$, the inequality$${\rm reg}(I^m)\geq {\rm reg}(I)+\gamma(I)(m-1)$$holds. Finally, in Theorem \ref{rint}, we investigate the relation between regularities of integral closure of certain powers of a monomial ideal. We prove that for a monomial ideal $I$ and for every pair of integers $s,m\geq 1$,$${\rm reg}(\overline{I^{sm}})\geq {\rm reg}(\overline{I^s})+\gamma(I)s(m-1).$$
%%%%%%%%%%%%%%%%%%%%%%%%%%%%%%%%%%%%%%%%%%%%%%%%%%%%%%%%%%%%%%%%%%%%%%%%%%

\section{Preliminaries} \label{sec2}

In this section, we provide the definitions and basic facts which will be used in the next sections.

Let $I$ be a monomial ideal of $S$. The (unique) set of minimal monomial generators of $I$ will be denoted by $G(I)$. Moreover, we set $\mu(I)=|G(I)|$. For an integer $m\geq 1$, the $m$-th {\it symbolic power} of $I$ is defined as$$I^{(m)}=\bigcap_{\frak{p}\in {\rm Min}(I)} {\rm Ker}(S\rightarrow (S/I^m)_{\frak{p}}),$$where ${\rm Min}(I)$ denotes the set of minimal prime ideals of $I$. An element $f \in S$ is
{\it integral} over $I$, if there exists an equation
$$f^k + c_1f^{k-1}+ \ldots + c_{k-1}f + c_k = 0 {\rm \ \ \ \ with} \ c_i\in I^i.$$
The set of elements $\overline{I}$ in $S$ which are integral over $I$ is the {\it integral closure}
of $I$. It is known that the integral closure of a monomial ideal $I\subset S$ is a monomial ideal
generated by all monomials $u \in S$ for which there exists an integer $k$ such that
$u^k\in I^k$ (see \cite[Theorem 1.4.2]{hh'}). The ideal $I$ is {\it integrally closed}, if $I = \overline{I}$. The Rees algebra of $I$ is the ring  $\mathcal{R}(I)=S[It]=\bigoplus_{k=0}^{\infty}I^kt^k$. Let $\frak{m}=(x_1,\ldots,x_n)$ be the maximal ideal of $S$. The Krull dimension of $\mathcal{R}(I)/
{{\frak{m}}\mathcal{R}(I)}$ is called the {\it analytic spread} of $I$ and will be denoted by $\ell(I)$.

A {\it simplicial complex} $\Delta$ on the set of vertices $[n]:=\{1,
\ldots,n\}$ is a collection of subsets of $[n]$ which is closed under
taking subsets; that is, if $F \in \Delta$ and $F'\subseteq F$, then also
$F'\in\Delta$. Every element $F\in\Delta$ is called a {\it face} of
$\Delta$. A {\it facet} of $\Delta$ is a maximal face
of $\Delta$ with respect to inclusion. The {\it link}
of $\Delta$ with respect to a face $F \in \Delta$, denoted by ${\rm lk_
{\Delta}}(F)$, is the simplicial complex ${\rm lk_{\Delta}}(F)=\{G
\subseteq [n]\setminus F\mid G\cup F\in \Delta\}$. By $\widetilde{H}_i(\Delta; \mathbb{K})$, we mean the $i$-th reduced homology of $\Delta$ with coefficients in $\mathbb{K}$.

Let $I$ be a squarefree monomial ideal of $S$. for a
subset $F\subseteq [n]$, we set $\mathbf{x}_F:=\prod_{i\in F}x_i$. The {\it
Stanley--Reisner} simplicial complex associated to $I$ is the simplicial complex $\Delta(I)=\{F\subseteq [n] : {\rm\bf x}_F\notin I\}$.

The following result from \cite{mnptv} has a key role in this paper. Here, $\mathbb{N}$ denotes the set of nonnegative integers. For a vector ${\rm\bf a}=(a_1, \cdots, a_n)\in \mathbb{N}^n$, we set ${\rm\bf x^a}:=x_1^{a_1}\cdots x_n^{a_n}$. Moreover, the {\it support} of ${\rm\bf a}$ is the set ${\rm supp}({\rm\bf a}):=\{i\in [n] : a_i\neq 0\}$. Furthermore, we set $|{\rm\bf a}|:=a_1+\cdots +a_n$.

\begin{lem} [\cite{mnptv}, Lemmata 2.8 and 2.13] \label{mv}
Let $I$ be a monomial ideal in $S$. Then
\begin{align*}
{\rm reg}(I)=\max\big\{|{\rm\bf a}|+i+1 :& \ {\rm\bf a}\in \mathbb{N}^n, i\geq 0, \tilde{H}_{i-1}({\rm lk}_{\Delta_{\rm\bf a}(I)}F ; \mathbb{K})\neq 0,\\ & \ {\rm for \ some} \ F\in \Delta_{\rm\bf a}(I) \ {\rm with} \ F \cap {\rm supp}({\rm\bf a})=\emptyset\big\},
\end{align*}
where $\Delta_{\rm\bf a}(I):=\Delta(\sqrt{I: {\rm\bf x^a}})$.
\end{lem}

For a monomial $u\in S$ and for any integer $j$ with $1\leq j\leq n$, the degree of $u$ with respect to $x_j$ is denoted by ${\rm deg}_{x_j}(u)$. We recall that for a monomial ideal $I$, the quantity $\gamma(I)$ is defined as$$\gamma(I):=\min\{{\rm deg}_{x_j}(u): 1\leq j\leq n, u\in G(I) \ {\rm and} \ x_j \ {\rm divides} \ u\}.$$

The following theorem is our first application of Lemma \ref{mv}. 

\begin{thm} \label{rrad}
Let $I$ be a proper nonzero monomial ideal of $S$. Then$${\rm reg}(I)\geq {\rm reg}(\sqrt{I})+(\gamma(I)-1)h,$$where $h$ denotes the height of $I$.
\end{thm}

\begin{proof}
Since $\sqrt{I}$ is a squarefree monomial ideal, it follows from Hochster's formula \cite[Theorem A.7.3]{hh'} that there exists an integer $i\geq 0$ such that ${\rm reg}(\sqrt{I})=i+1$ and $\tilde{H}_{i-1}({\rm lk}_{\Delta(\sqrt{I})}F; \mathbb{K})\neq 0$, for some face $F\in \Delta(\sqrt{I})$. Note that $F$ is contained in a facet of $\Delta(\sqrt{I})$. Since $I$ and $\sqrt{I}$ have the same height, it follows from \cite[Lemma 1.5.4]{hh'} that the cardinality of $[n]\setminus F$ is at least $h$. So, there are distinct elements $r_1, \ldots, r_h\in [n]\setminus F$. Let ${\rm\bf b}\in \mathbb{N}^n$ be the vector with ${\rm\bf x^b}=x_{r_1}^{\gamma(I)-1}\cdots x_{r_h}^{\gamma(I)-1}$. Then the definition of $\gamma(I)$ implies that $\sqrt{I:{\rm\bf x^b}}=\sqrt{I}$. This means that $\Delta_{\rm\bf b}(I)=\Delta(\sqrt{I})$. Therefore,$$\tilde{H}_{i-1}({\rm lk}_{\Delta_{\rm\bf b}(I)}F; \mathbb{K})=\tilde{H}_{i-1}({\rm lk}_{\Delta(\sqrt{I})}F; \mathbb{K})\neq 0.$$Hence, we conclude from Lemma \ref{mv} that
\begin{align*}
&{\rm reg}(I)\geq |{\rm\bf b}|+i+1=(\gamma(I)-1)h+i+1=(\gamma(I)-1)h+{\rm reg}(\sqrt{I}).
\end{align*}
\end{proof}

%%%%%%%%%%%%%%%%%%%%%%%%%%%%%%%%%%%%%%%%%%%%%%%%%%%%%%%%%%%%%%%%%%%%%%%%%%

\section{Symbolic powers} \label{sec3}

In this section, we study the relation between regularities of certain symbolic powers of a squarefree monomial ideal. The following theorem is the main result of this section.

\begin{thm} \label{sym}
Let $I$ be a nonzero proper squarefree monomial ideal of $S$. Suppose that $m$ and $k$ are positive integers. Then for every integer $j$ with $m-k\leq j\leq m$, we have$${\rm reg}(I^{(km+j)})\geq {\rm reg}(I^{(m)})+(k-1)m+j.$$
\end{thm}

\begin{proof}
By Lemma \ref{mv}, there exists a vector ${\rm\bf a}\in \mathbb{N}^n$ and an integer $i\geq 0$ such that ${\rm reg}(I^{(m)})=|{\rm\bf a}|+i+1$ and $\tilde{H}_{i-1}({\rm lk}_{\Delta_{\rm\bf a}(I^{(m)})}F; \mathbb{K})\neq 0$, for some face $F\in \Delta_{\rm\bf a}(I^{(m)})$ with $F\cap {\rm supp}({\rm\bf a})=\emptyset$. We divide the proof into the following two cases.

\vspace{0.3cm}
{\bf Case 1.} Assume that $|{\rm\bf a}|\leq m-1$. Then we conclude from \cite[Lemma 2.11]{mv} that $\Delta_{\rm\bf a}(I^{(m)})=\Delta(I)$. If $F=[n]$, it follows from $F\in \Delta_{\rm\bf a}(I^{(m)})=\Delta(I)$ that $I=0$ which is a contradiction. Hence, $F\neq [n]$. So, there is an element $r\in [n]\setminus F$. Let ${\rm\bf b}\in \mathbb{N}^n$ be the vector with ${\rm\bf x^b}=x_r^{(k-1)m+j}{\rm\bf x^a}$. Then$$|{\rm\bf b}|=(k-1)m+j+|{\rm\bf a}|\leq (k-1)m+j+m-1=km+j-1.$$Thus, \cite[Lemma 2.11]{mv} yields that $\Delta_{\rm\bf b}(I^{(km+j)})=\Delta(I)$. Therefore,$$\tilde{H}_{i-1}({\rm lk}_{\Delta_{\rm\bf b}(I^{(km+j)})}F; \mathbb{K})=\tilde{H}_{i-1}({\rm lk}_{\Delta(I)}F; \mathbb{K})=\tilde{H}_{i-1}({\rm lk}_{\Delta_{\rm\bf a}(I^{(m)})}F; \mathbb{K})\neq 0.$$Hence, Lemma \ref{mv} implies that
\begin{align*}
&{\rm reg}(I^{(km+j)})\geq |{\rm\bf b}|+i+1=(k-1)m+j+|{\rm\bf a}|+i+1\\ &=(k-1)m+j+{\rm reg}(I^{(m)}).
\end{align*}

\vspace{0.3cm}
{\bf Case 2.} Assume that $|{\rm\bf a}|\geq m$. We claim that$$\sqrt{I^{(m)}:{\rm\bf x^a}}=\sqrt{I^{(km+j)}:{\rm\bf x}^{(k+1){\rm\bf a}}}.$$

Note that for a monomial $u\in S$, we have $u\in \sqrt{I^{(m)}:{\rm\bf x^a}}$ if and only if there is a positive integer $t$ such that $u^t{\rm\bf x^a}\in I^{(m)}$. Using \cite[Lemma 2.2]{s12}, this is equivalent to say that $u^{(k+1)t}{\rm\bf x}^{(k+1){\rm\bf a}}\in I^{(km+j)}$ which means that $u\in \sqrt{I^{(km+j)}:{\rm\bf x}^{(k+1){\rm\bf a}}}$. This proves the claim.

Set ${\rm\bf c}:=(k+1){\rm\bf a}$. It follows from the claim that $\Delta_{\rm\bf a}(I^{(m)})=\Delta_{\rm\bf c}(I^{(km+j)})$. Consequently,$$\tilde{H}_{i-1}({\rm lk}_{\Delta_{\rm\bf c}(I^{(km+j)})}F; \mathbb{K})=\tilde{H}_{i-1}({\rm lk}_{\Delta_{\rm\bf a}(I^{(m)})}F; \mathbb{K})\neq 0.$$Therefore, we deduce from Lemma \ref{mv} that
\begin{align*}
&{\rm reg}(I^{(km+j)})\geq |{\rm\bf c}|+i+1=(k+1)|{\rm\bf a}|+i+1\geq km+|{\rm\bf a}|+i+1\\ &\geq(k-1)m+j+{\rm reg}(I^{(m)}),
\end{align*}
where the second inequality follows from $|{\rm\bf a}|\geq m$ and the third one is a consequence of $j\leq m$.
\end{proof}

As an immediate consequence of Theorem \ref{sym}, we obtain the following corollary.

\begin{cor} \label{corsym}
Let $I$ be a nonzero proper squarefree monomial ideal of $S$.
\begin{itemize}
\item[(i)] For every pair of positive integers $k,m$, we have$${\rm reg}(I^{(km)})\geq {\rm reg}(I^{(m)})+(k-1)m.$$
\item[(ii)] ${\rm reg}(I^{(3)})\geq {\rm reg}(I^{(2)})+1$.
\end{itemize}
\end{cor}

We recall that a special case of Corollary \ref{corsym}(i) (i.e., when $m=1$) is obtained by Minh and Vu \cite{mv}.

Using Theorem \ref{sym}, we are also able to provide an upper bound for the regularity of certain quotients of symbolic powers of a squarefree monomial ideal.

\begin{cor} \label{symq}
Let $I$ be a nonzero proper squarefree monomial ideal of $S$. Suppose that $m$ and $k$ are positive integers. Then for every integer $j$ with $m-k\leq j\leq m$, we have$${\rm reg}(I^{(m)}/I^{(km+j)})\leq {\rm reg}(I^{(km+j)})-1.$$
\end{cor}

\begin{proof}
It follows from the assumptions that $(k-1)m+j\geq 0$. If $(k-1)m+j=0$, then $km+j=m$. In this case $I^{(m)}/I^{(km+j)}$ is the zero  module and there is nothing to prove. So, assume that $(k-1)m+j\geq 1$. Consider the following short exact sequence.
$$0\longrightarrow I^{(km+j)}\longrightarrow I^{(m)}\longrightarrow I^{(m)}/I^{(km+j)}\rightarrow 0$$
It follows from \cite[Corollary 18.7]{p'} that$${\rm reg}(I^{(m)}/I^{(km+j)})\leq\max\big\{{\rm reg}(I^{(m)}), {\rm reg}(I^{(km+j)})-1\big\}.$$Since $(k-1)m+j\geq 1$, the assertion follows by applying Theorem \ref{sym}.
\end{proof}

%%%%%%%%%%%%%%%%%%%%%%%%%%%%%%%%%%%%%%%%%%%%%%%%%%%%%%%%%%%%%%%%%%%%%%%%%%

\section{Integral closure of powers} \label{sec4}

In this section, we study the regularity of integral closure of powers of a monomial ideal $I$. Theorem \ref{rnormal1} is the first main result of this section which provides an upper bound for the regularity of $\overline{I}$ in terms of the regularity of certain powers of $I$. We need the following two lemmas in the proof of Theorem \ref{rnormal1}.

\begin{lem} \label{colonns}
Let $I$ be a proper nonzero monomial ideal of $S$. Then for every integer $s\geq 1$ and for every vector ${\rm\bf a}\in \mathbb{N}^n$ with $|{\rm\bf a}|\leq \gamma(I)s-1$, we have $\Delta_{\rm\bf a}(I^s)=\Delta(\sqrt{I})$.
\end{lem}

\begin{proof}
It is enough to prove that $\sqrt{I^s:{\rm\bf x^a}}=\sqrt{I}$. Trivially, we have $\sqrt{I}\subseteq\sqrt{I^s:{\rm\bf x^a}}$. On the other hand, since $|{\rm\bf a}|\leq \gamma(I)s-1$, it is easy to check that $(I^s:{\rm\bf x^a})\subseteq \sqrt{I}$ and hence, $\sqrt{I^s:{\rm\bf x^a}}\subseteq\sqrt{I}$.
\end{proof}

\begin{lem} \label{colonnsi}
Let $I$ be a proper nonzero monomial ideal of $S$. Then for every integer $s\geq 1$ and for every vector ${\rm\bf a}\in \mathbb{N}^n$ with $|{\rm\bf a}|\leq \gamma(I)s-1$, we have $\Delta_{\rm\bf a}(\overline{I^s})=\Delta(\sqrt{I})$.
\end{lem}

\begin{proof}
It is enough to prove that $\sqrt{\overline{I^s}:{\rm\bf x^a}}=\sqrt{I}$. Trivially, we have $\sqrt{I}\subseteq\sqrt{\overline{I^s}:{\rm\bf x^a}}$. To prove the reverse inclusion, let $u$ be a monomial in $\sqrt{\overline{I^s}:{\rm\bf x^a}}$. Then $u^t{\rm\bf x^a}\in \overline{I^s}$, for some integer $t\geq 1$. Hence, there is an integer $k\geq 1$ such that $u^{kt}{\rm\bf x}^{k{\rm\bf a}}\in I^{ks}$. Therefore, $u\in\sqrt{I^{ks}: {\rm\bf x}^{k{\rm\bf a}}}$. Since$$|k{\rm\bf a}|\leq k(\gamma(I)s-1)\leq k\gamma(I)s-1,$$it follows from Lemma \ref{colonns} that $u \in \sqrt{I}$. Consequently, $\sqrt{\overline{I^s}:{\rm\bf x^a}}\subseteq \sqrt{I}$.
\end{proof}

We are now ready to prove the first main result of this section.

\begin{thm} \label{rnormal1}
Let $I$ be a proper nonzero monomial ideal of $S$ with analytic spread $\ell=\ell(I)$. Set $s=\mu(\overline{I^{\ell-1}})!$. Then for every integer $m\geq 1$, we have$${\rm reg}(\overline{I})\leq {\rm reg}(I^{sm})-\gamma(I)(sm-1).$$
\end{thm}

\begin{proof}
By Lemma \ref{mv}, there exists a vector ${\rm\bf a}\in \mathbb{N}^n$ and an integer $i\geq 0$ such that ${\rm reg}(\overline{I})=|{\rm\bf a}|+i+1$ and $\tilde{H}_{i-1}({\rm lk}_{\Delta_{\rm\bf a}(\overline{I})}F; \mathbb{K})\neq 0$, for some face $F\in \Delta_{\rm\bf a}(\overline{I})$ with $F\cap {\rm supp}({\rm\bf a})=\emptyset$. We divide the proof into the following two cases.

\vspace{0.3cm}
{\bf Case 1.} Assume that $|{\rm\bf a}|\leq \gamma(I)-1$. Then it follows from Lemma \ref{colonnsi} that $\Delta_{\rm\bf a}(\overline{I})=\Delta(\sqrt{I})$. If $F=[n]$, it follows from $F\in \Delta_{\rm\bf a}(\overline{I})=\Delta(\sqrt{I})$ that $I=0$ which is a contradiction. Hence, $F\neq [n]$. So, there is an element $r\in [n]\setminus F$. Let ${\rm\bf b}\in \mathbb{N}^n$ be the vector with ${\rm\bf x^b}=x_r^{\gamma(I)(sm-1)}{\rm\bf x^a}$. Then$$|{\rm\bf b}|=\gamma(I)(sm-1)+|{\rm\bf a}|\leq \gamma(I)sm-1.$$Consequently, Lemma \ref{colonns} yields that $\Delta_{\rm\bf b}(I^{sm})=\Delta(\sqrt{I})$. Therefore,$$\tilde{H}_{i-1}({\rm lk}_{\Delta_{\rm\bf b}(I^{sm})}F; \mathbb{K})=\tilde{H}_{i-1}({\rm lk}_{\Delta(\sqrt{I})}F; \mathbb{K})=\tilde{H}_{i-1}({\rm lk}_{\Delta_{\rm\bf a}(\overline{I})}F; \mathbb{K})\neq 0.$$Hence, we conclude from Lemma \ref{mv} that
\begin{align*}
&{\rm reg}(I^{sm})\geq |{\rm\bf b}|+i+1=\gamma(I)(sm-1)+|{\rm\bf a}|+i+1=\gamma(I)(sm-1)+{\rm reg}(\overline{I}).
\end{align*}

\vspace{0.3cm}
{\bf Case 2.} Assume that $|{\rm\bf a}|\geq\gamma(I)$. We claim that$$\sqrt{\overline{I}:{\rm\bf x^a}}=\sqrt{I^{sm}:{\rm\bf x}^{(sm){\rm\bf a}}}.$$

To prove the claim, let $u$ be a monomial in $\sqrt{I^{sm}:{\rm\bf x}^{(sm){\rm\bf a}}}$. Then $u^t{\rm\bf x}^{(sm){\rm\bf a}}\in I^{sm}$, for some integer $t\geq 1$. Consequently,  $u^{t(sm)}{\rm\bf x}^{(sm){\rm\bf a}}\in I^{sm}$. This implies that $u^t{\rm\bf x^a}\in \overline{I}$ and so, $u\in \sqrt{\overline{I}:{\rm\bf x^a}}$. Therefore,$$\sqrt{I^{sm}:{\rm\bf x}^{(sm){\rm\bf a}}}\subseteq \sqrt{\overline{I}:{\rm\bf x^a}}.$$To prove the reverse inclusion, suppose that $G(\overline{I})=\{v_1, \ldots, v_p\}$ is the set of minimal monomial generators of $\overline{I}$. Using \cite[Lemma 4.4]{s13}, for each $j=1, \ldots, p$ there is an integer $k_j\leq \mu(\overline{I^{\ell-1}})$, such that $v_j^{k_j}\in I^{k_j}$. As $k_j$ divides $s$ for every $j$, we conclude that $v_j^s\in I^s$. Hence, for each monomial $v\in \overline{I}$, we have $v^{sm}\in I^{sm}$. Now assume that $w$ is a monomial in $\sqrt{\overline{I}:{\rm\bf x^a}}$. Then $w^l{\rm\bf x^a}\in \overline{I}$, for some integer $l\geq 1$. The above discussion implies that $w^{(sm)l}{\rm\bf x}^{(sm){\rm\bf a}}\in I^{sm}$. Thus, $w\in \sqrt{I^{sm}:{\rm\bf x}^{(sm){\rm\bf a}}}$. This shows that$$\sqrt{\overline{I}:{\rm\bf x^a}}\subseteq\sqrt{I^{sm}:{\rm\bf x}^{(sm){\rm\bf a}}},$$and completes the proof of the claim.

Set ${\rm\bf c}:=(sm){\rm\bf a}$. It follows from the claim that $\Delta_{\rm\bf a}(\overline{I})=\Delta_{\rm\bf c}(I^{sm})$. Consequently,$$\tilde{H}_{i-1}({\rm lk}_{\Delta_{\rm\bf c}(I^{sm})}F; \mathbb{K})=\tilde{H}_{i-1}({\rm lk}_{\Delta_{\rm\bf a}(\overline{I})}F; \mathbb{K})\neq 0.$$Therefore, we deduce from Lemma \ref{mv} that
\begin{align*}
&{\rm reg}(I^{sm})\geq |{\rm\bf c}|+i+1=(sm)|{\rm\bf a}|+i+1=(sm-1)|{\rm\bf a}|+|{\rm\bf a}|+i+1\\ &\geq(sm-1)\gamma(I)+{\rm reg}(\overline{I}),
\end{align*}
where the last inequality follows from $|{\rm\bf a}|\geq\gamma(I)$.
\end{proof}

\begin{rem} \label{remint}
Suppose $I$ is a monomial ideal of $S$. Let $s\geq 1$ be an integer with the property that for any monomial $u\in G(\overline{I})$, we have $u^s\in I^s$. Then the proof of Theorem \ref{rnormal1} shows that for any integer $m\geq 1$,$${\rm reg}(\overline{I})\leq {\rm reg}(I^{sm})-\gamma(I)(sm-1).$$
\end{rem}

The following theorem shows that the regularity of the $m$-th power of an integrally closed monomial ideal $I$ is at least ${\rm reg}(I)+\gamma(I)(m-1)$. We recall that a special case of this result (i.e., when $I$ is a squarefree monomial ideal) is proved by Minh and Vu \cite{mv}.

\begin{thm} \label{rintc}
Let $I$ be a proper nonzero integrally closed monomial ideal of $S$. Then for every integer $m\geq 1$, we have$${\rm reg}(I^m)\geq {\rm reg}(I)+\gamma(I)(m-1).$$
\end{thm}

\begin{proof}
The same argument as in the proof of Theorem \ref{rnormal1} implies the assertion. The only difference is that since $I$ is integrally closed, for any monomial $u\in \overline{I}$, we have $u\in I$. So, in the proof of Theorem \ref{rnormal1}, one may replace $s$ by $1$ (see also Remark \ref{remint}).
\end{proof}

Let $I$ be a monomial ideal of $S$. In Theorem \ref{rint}, we study the relation between regularities of integral closure of certain symbolic powers of $I$. More precisely, we prove that for any pair of integers $s,m\geq 1$, the regularity of $\overline{I^{sm}}$ is bounded below by ${\rm reg}(\overline{I^s})+\gamma(I)s(m-1)$.

\begin{thm} \label{rint}
Let $I$ be a proper nonzero monomial ideal of $S$. Then for every pair of integers $s,m\geq 1$, we have$${\rm reg}(\overline{I^{sm}})\geq {\rm reg}(\overline{I^s})+\gamma(I)s(m-1).$$
\end{thm}

\begin{proof}
By Lemma \ref{mv}, there exists a vector ${\rm\bf a}\in \mathbb{N}^n$ and an integer $i\geq 0$ such that ${\rm reg}(\overline{I^s})=|{\rm\bf a}|+i+1$ and $\tilde{H}_{i-1}({\rm lk}_{\Delta_{\rm\bf a}(\overline{I^s})}F; \mathbb{K})\neq 0$, for some face $F\in \Delta_{\rm\bf a}(\overline{I})$ with $F\cap {\rm supp}({\rm\bf a})=\emptyset$. We divide the proof into the following two cases.

\vspace{0.3cm}
{\bf Case 1.} Assume that $|{\rm\bf a}|\leq \gamma(I)s-1$. Then we conclude from Lemma \ref{colonnsi} that $\Delta_{\rm\bf a}(\overline{I^s})=\Delta(\sqrt{I})$. If $F=[n]$, it follows from $F\in \Delta_{\rm\bf a}(\overline{I^s})=\Delta(\sqrt{I})$ that $I=0$ which is a contradiction. Hence, $F\neq [n]$. So, there is an element $r\in [n]\setminus F$. Let ${\rm\bf b}\in \mathbb{N}^n$ be the vector with ${\rm\bf x^b}=x_r^{\gamma(I)s(m-1)}{\rm\bf x^a}$. As $|{\rm\bf a}|\leq \gamma(I)s-1$, we have$$|{\rm\bf b}|=\gamma(I)s(m-1)+|{\rm\bf a}|\leq \gamma(I)sm-1.$$Thus, Lemma \ref{colonnsi} yields that $\Delta_{\rm\bf b}(\overline{I^{sm}})=\Delta(\sqrt{I})$. Therefore,$$\tilde{H}_{i-1}({\rm lk}_{\Delta_{\rm\bf b}(\overline{I^{sm}})}F; \mathbb{K})=\tilde{H}_{i-1}({\rm lk}_{\Delta_{\rm\bf a}(\sqrt{I})}F; \mathbb{K})=\tilde{H}_{i-1}({\rm lk}_{\Delta_{\rm\bf a}(\overline{I^s})}F; \mathbb{K})\neq 0.$$Hence, Lemma \ref{mv} implies that
\begin{align*}
&{\rm reg}(\overline{I^{sm}})\geq |{\rm\bf b}|+i+1=\gamma(I)s(m-1)+|{\rm\bf a}|+i+1=\gamma(I)s(m-1)+{\rm reg}(\overline{I^s}).
\end{align*}

\vspace{0.3cm}
{\bf Case 2.} Assume that $|{\rm\bf a}|\geq \gamma(I)s$. We claim that$$\sqrt{\overline{I^s}:{\rm\bf x^a}}=\sqrt{\overline{I^{sm}}:{\rm\bf x}^{m{\rm\bf a}}}.$$

To prove the claim, let $u$ be a monomial in $\sqrt{\overline{I^{sm}}:{\rm\bf x}^{m{\rm\bf a}}}$. Then $u^t{\rm\bf x}^{m{\rm\bf a}}\in \overline{I^{sm}}$, for some integer $t\geq 1$. Hence, there exists an integer $k\geq 1$ such that $u^{kt}{\rm\bf x}^{(km){\rm\bf a}}\in I^{ksm}$. Consequently,  $u^{ktm}{\rm\bf x}^{(km){\rm\bf a}}\in I^{ksm}$. This implies that $u^t{\rm\bf x}^{\rm\bf a}\in \overline{I^s}$ and So, $u\in \sqrt{\overline{I^s}:{\rm\bf x^a}}$. Therefore,$$\sqrt{\overline{I^{sm}}:{\rm\bf x}^{m{\rm\bf a}}}\subseteq \sqrt{\overline{I^s}:{\rm\bf x^a}}.$$To prove the reverse inclusion, let $v$ be a monomial in $\sqrt{\overline{I^s}:{\rm\bf x^a}}$. Then $v^l{\rm\bf x^a}\in \overline{I^s}$, for some integer $l\geq 1$. Thus, $v^{ml}{\rm\bf x}^{m{\rm\bf a}}\in \overline{I^{ms}}$. Hence, $v\in \sqrt{\overline{I^{sm}}:{\rm\bf x}^{m{\rm\bf a}}}$. This shows that$$\sqrt{\overline{I^s}:{\rm\bf x^a}}\subseteq\sqrt{\overline{I^{sm}}:{\rm\bf x}^{m{\rm\bf a}}},$$and completes the proof of the claim.

Set ${\rm\bf c}:=m{\rm\bf a}$. It follows from the claim that $\Delta_{\rm\bf a}(\overline{I^s})=\Delta_{\rm\bf c}(\overline{I^{sm}})$. Consequently,$$\tilde{H}_{i-1}({\rm lk}_{\Delta_{\rm\bf c}(\overline{I^{sm}})}F; \mathbb{K})=\tilde{H}_{i-1}({\rm lk}_{\Delta_{\rm\bf a}(\overline{I^s})}F; \mathbb{K})\neq 0.$$Therefore, we deduce from Lemma \ref{mv} that
\begin{align*}
&{\rm reg}(\overline{I^{sm}})\geq |{\rm\bf c}|+i+1=m|{\rm\bf a}|+i+1=(m-1)|{\rm\bf a}|+|{\rm\bf a}|+i+1\\ &\geq (m-1)\gamma(I)s+{\rm reg}(\overline{I^s}),
\end{align*}
where the last inequality follows from $|{\rm\bf a}|\geq \gamma(I)s$.
\end{proof}

Using Theorem \ref{rint}, we are able to provide an upper bound for the regularity of certain quotients of integral closure of powers of a monomial ideal.

\begin{cor} \label{intq}
Let $I$ be a nonzero proper monomial ideal of $S$. Then for every pair of integers $s,m\geq 1$, we have$${\rm reg}(\overline{I^s}/\overline{I^{sm}})\leq {\rm reg}(\overline{I^{sm}})-1.$$
\end{cor}

\begin{proof}
It follows from the assumptions that $\gamma(I)s(m-1)\geq 0$. If $\gamma(I)s(m-1)=0$, then $m=1$. In this case $\overline{I^s}/\overline{I^{sm}}$ is the zero module and there is nothing to prove. So, assume that $\gamma(I)s(m-1)\geq 1$. Consider the following short exact sequence.
$$0\longrightarrow \overline{I^{sm}}\longrightarrow \overline{I^s}\longrightarrow \overline{I^s}/\overline{I^{sm}}\rightarrow 0$$
It follows from \cite[Corollary 18.7]{p'} that$${\rm reg}(\overline{I^s}/\overline{I^{sm}})\leq\max\big\{{\rm reg}(\overline{I^s}), {\rm reg}(\overline{I^{sm}})-1\big\}.$$Since $\gamma(I)s(m-1)\geq 1$, the assertion follows by applying Theorem \ref{rint}.
\end{proof}

%%%%%%%%%%%%%%%%%%%%%%%%%%%%%%%%%%%%%%%%%%%%%%%%%%%%%%%%%%%%%%%%%%%%%%%%%%

\section*{Acknowledgment}

The author is supported by the FAPA grant from the Universidad de los Andes.

%%%%%%%%%%%%%%%%%%%%%%%%%%%%%%%%%%%%%%%%%%%%%%%%%%%%%%%%%%%%%%%%%%%%%%%%%%

%%%%%%%%%%%%%%%%%%%%%%%%%%%%%%%%%%%%%%%%%%%%%%%%%%%%%%%%%%%%%%%%%%%%%%%%%%

\end{document}